\documentclass{amsart}
\usepackage{amsthm,amsmath,amsfonts}
\usepackage{mathrsfs}
\usepackage[colorlinks=false,linktocpage=true]{hyperref}
\usepackage{tocvsec2}
\usepackage{latexsym}
\newtheorem{thm}{Theorem}[section]

\newtheorem{lem}{Lemma}[section]
\newtheorem{prop}{Proposition}[section]

\theoremstyle{definition}
\newtheorem{defn}{Definition}[section]

\newtheorem{rem}{Remark}[section]
\newcommand{\thmref}[1]{Theorem~\ref{#1}}

\newcommand{\lemref}[1]{Lemma~\ref{#1}}

\newcommand{\defnref}[1]{Definition~\ref{#1}}
\newcommand{\remref}[1]{Remark~\ref{#1}}
\newcommand{\eqnref}[1]{{\rm (\ref{#1})}}

\def\ome{\omega}

\def\be#1{\begin{equation}\label{#1}}
\def\ee{\end{equation}}
\def\cal{\mathcal}
\def\E#1{{\rm E}[#1]}

\newcommand \R {{\bf R}}
\newcommand \N {{\bf N}}

\newcommand \teta{\theta}

\newcommand \eps{\epsilon}

\newcommand \ba {|}
\newcommand \Phie {\Phi (r+s,\theta _{-s}\omega )u_0}
\newcommand \Phiev {\Phi (r+s,\theta _{-s}\omega )v_0}

\def\ome{\omega}
\def\dis {\displaystyle}

\def\dstyle{\displaystyle}
\def\dis{\displaystyle}
\def\eps{\epsilon}

\def\A{{\cal A}}
\def\at{{\cal A}}
\def\atw{{\cal A}(\omega)}
\def\eh{e_{AC}(f,u_0)}
\def\P{{\mathbb P}}
\def\EP{{\mathbb E}_{\P}}
\def\E{{\mathbb E}}

\def\Wm{{\cal W}}

\def\vfi{\varphi}

\def\vfix{\vfi(x)}

\def\A{{\cal A}}
\def\at{{\cal A}}
\def\atw{{\cal A}(\omega)}
\def\B{{\mathbb B}}
\def\L{{\mathbb L}}

\def\N{{\mathbb N}}
\def\R{{\mathbb R}}
\def\Rd{{{\mathbb R}^d}}
\def\Rp{{\mathbb R}_+}

\def\pspace{(\Omega, {\cal F}, \P)}
\def\pspaceth{(\Omega, {\cal F}, \P, (\theta _t)_{t\in \R})}
\def\filspace{(\Omega, {\cal F}, \{{{\cal F}}_t\},\P)}
\def\Ft{{\cal F}_t}
\def\bP{{\textbf{P}}}
\def\bQ{{\textbf{Q}}}
\def\bI{{\textbf{I}}}
\def\oA{{\Delta}}
\def\bA{{\mathbb A}}
\def\vfi{\varphi}

\begin{document}
\title[Semimartingale attractors for
space-time white noise Allen-Cahn SPDEs]{SEMIMARTINGALE ATTRACTORS FOR ALLEN-CAHN SPDEs DRIVEN
BY SPACE-TIME WHITE NOISE I: EXISTENCE AND FINITE DIMENSIONAL
ASYMPTOTIC BEHAVIOUR}
\author{H. ALLOUBA}
\address{Department of Mathematical Sciences \\
 Kent State University, \ Kent,  \
  OH 44242,\
  USA\\
allouba@mcs.kent.edu}
\author{J.A. LANGA}
\address{Departamento de Ecuaciones Diferenciales y
An\'{a}lisis
Num\'{e}rico, \\
Universidad de Sevilla, \ Apdo. de Correos 1160,
41080-Sevilla, Spain\\
langa@us.es}
\keywords{semimartingale attractors; stochastic Allen-Cahn
equations; space-time white noise}
\subjclass{37H10; 37H15; 35B42}
\date{April 26, 2004}
\maketitle
\begin{abstract}
We delve deeper into the study of semimartingale attractors that
we recently introduced in Allouba and Langa \cite{AL0}. In this
article we focus on second order SPDEs of the Allen-Cahn type.
After proving existence, uniqueness, and detailed regularity
results for our SPDEs and a corresponding random PDE of Allen-Cahn
type, we prove the existence of semimartingale global attractors
for these equations.  We also give some results on the finite
dimensional asymptotic behavior of the solutions.  In particular,
we show the finite fractal dimension of this random attractor and
give a result on determining modes, both in the forward and the
pullback sense.
\end{abstract}

\section{Introduction and organization of the article}
The analysis of qualitative properties of ordinary and partial
differential equations is the key point in dynamical system
theory. When a phenomenon from Physics, Chemistry, Biology,
Economics can be described by a system of differential equations
(in which the existence of global solutions can be assured), one
of the most interesting problems is to describe the asymptotic
behavior of the system when time grows to infinity. The study of
the asymptotic dynamics of the system gives us relevant
information about ``the future'' of the phenomenon described in
the model. In this context, the concept of \textit{global
attractor} has become a very useful tool to describe the long-time
behavior of many important differential equations (see, among
others, Ladyzhenskaya \cite{Ladyzhenskaya}, Babin and Vishik
\cite{Babin}, Hale \cite{Hale}, Temam \cite{T}, Robinson
\cite{Ro01}). A new difficulty appears when a random term is added
to the deterministic equation, a white noise for instance, and the
resulting stochastic partial differential equation must be treated
in a different way. Firstly, the equation becomes non-autonomous,
which makes necessary the introduction of a two-sided time
dependent process instead of a semigroup. Moreover, the strong
dependence on the random term adds another difficulty. The rapidly
growing theory of \textit{random dynamical systems} (Arnold
\cite{Ar98}) has become the appropriate tool for the study of many
important random and stochastic equations. In this framework,
Crauel and Flandoli \cite{Cr94} (see also Schmalfuss \cite{Sc92})
introduced the concept of a random attractor as a proper
generalization  of the corresponding deterministic global
attractor. The theory of random attractors is turning out to be
very helpful in the understanding of the long-time dynamics of
some stochastic ordinary and partial differential equations. On
the other hand, one of the most important results in the theory of
global attractors for deterministic PDEs claims that the fractal,
and so the Hausdorff, dimension of this set is finite (Constantin
and Foias \cite{CoFo}, Contantin et al. \cite{CoFoTe},
Ladyzhenskaya \cite{Lad}; see also the books of Temam \cite{T} and
Robinson \cite{Ro01}). That is, although the trajectories depend
on an infinite number of degrees of freedom, the finite
dimensionality of the attractors leads to the idea that the
asymptotic behavior can be described by a finite number of
time-dependent coordinates. This makes, for example, really
interesting the study of the dynamics on the global attractor.
There are also  some results which generalize the
finite-dimensionality of attractors to the stochastic case
(Debussche \cite{De97}, \cite{De98}).
\par
In this paper we show how all the theory of finite dimensional
random attractors can be generalized to the situation in which the
partial differential equation is affected by a space-time white
noise, and we characterize this randomness in the attractor as one
coming from semimartingale-type solutions (see \defnref{sa}). Some
of these results were recently sketched in Allouba and Langa
\cite{AL0}.   Here, we prove in details the existence of a finite
dimensional random attractor associated to the random dynamical
system corresponding to a space-time white noise driven stochastic
PDE of Allen-Cahn type; and we give a determining modes result for
such a SPDE, both in the forward and pullback sense.    In the
course of our proof, we also give detailed proofs and discussions
of existence, uniqueness, and regularity (both weak and strong)
results for our SPDE as well as for an associated Allen-Cahn type
random PDE. The lack of regularity caused by our driving
space-time white noise causes several difficulties in the SPDEs
we study.  These difficulties are not present in the traditional
case of noises that are only white in time (see \remref{reg} and
\remref{sae} below).
\par
Before spelling out the organization of this paper, we wanted to
highlight two key features of this work:
\begin{enumerate}
\item[i)]  Our solutions are weak semimartingales
(see \defnref{sa} and Section 3.4 below), and this characterizes
the randomness in our attractors as one coming from some type of
semimartingale solutions (not simply random processes); thus we
call our random attractors semimartingale attractors. This
characterization is crucial and will lead to several new
stochastic analytic aspects of these random attractors, like the
notion of semimartingale decomposition of semimartingale
attractors (e.g., \cite{ALII}).
\item[ii)]  As in Walsh \cite{WA}, we regard space-time
white noise as a continuous orthogonal martingale measure, which
we think will lead to a richer structure of the noise, and so to
new aspects of the SPDE under consideration, even compared to
cylindrical noise. One such aspect is the notion of semimartingale
measure attractors (to which we devote a separate paper), which is
built upon the notion of semimartingale measure introduced in
Allouba \cite{A298}.
\end{enumerate}

\par
The paper is organized as follows: in the next Section we write
the general theory of random attractors and give the definition of
weak semimartingales; Section 3 develops the existence,
uniqueness, and regularity (both weak and strong) of solutions for
a stochastic PDE of Allen-Cahn type with space-time white noise
and for a corresponding random PDE; we follow by proving the
existence of a semimartingale attractor associated to these
equations. Finally, we show the dependence of the asymptotic
behavior of the model on a finite number of degrees of freedom, by
proving, with probability one, the finite fractal dimensionality
of the semimartingale attractor  and some results on determining
modes, both in the forward and the pullback sense. Some
conclusions are then given, placing the results here in the
context of our ongoing research program. We also include some
technical results in a final Appendix.   Throughout this article
we will denote by $K$ a constant that may change its value from
line to line.

\section{Semimartigale global attractors}
\subsection{Definitions}
Proceeding toward a precise statement of our results, let us
recall some definitions associated with random attractors.  Let
$\pspace$ be a probability space and $\{\theta_t: \Omega \to
\Omega , \, t\in \R\}$ a family of measure preserving
transformations such that $(t,\omega )\mapsto \theta _t\omega $ is
measurable, $\theta _0={\rm {id}}$, $\theta _{t+s}=\theta _t\theta
_s$, for all $s,t \in \R $. The flow $\theta _t $ together with
the probability space $\pspaceth$ is called a {\it measurable
dynamical system}. Furthermore, we suppose that the shift $\theta
_t$ is ergodic. \par A {\it random dynamical system} (RDS) (Arnold
\cite{Ar98}) on a complete metric (or Banach) space $(\B,d)$ with
Borel $\sigma$-algebra ${\cal B}$, over $\theta $ on $\pspace$ is
a measurable map $ \Rp\times \Omega \times\B\ni(t,\omega , \xi)
\mapsto  \Phi (t, \omega )\xi\in\B$ such that $\P$--a.s.
\begin{enumerate}
\renewcommand{\labelenumi}{\roman{enumi})}
\item  $\Phi (0,\omega )={\rm {id} }\qquad (\mbox{on }\B)$
\item  $\Phi (t+s,\omega )= \Phi (t, \theta _s \omega )\circ \Phi (s, \omega ), \: \forall \, t,s \in \Rp \quad
\mbox{(cocycle property)}.$
\end{enumerate}
A RDS is continuous (differentiable) if $\Phi (t,\omega ):\B\to\B$ is continuous (differentiable).    A random set
$K(\omega )\subset\B$ is said to {\it absorb} the set $D\subset\B$ if  there exists a random time $t_D(\omega )$ such that
$$t\ge t_D(\omega )\rightarrow \Phi (t, \theta _{-t} \omega )D \subset K(\omega ), \P\mbox{--a.s.}$$
$K(\omega )$ is forward invariant if $\Phi (t, \omega ) K(\omega
)\subseteq K(\theta _t \omega), \mbox{ for all } t\in \Rp,
\P\mbox{--a.s.}$ Now, let ${\rm {dist}}(\cdot,\cdot)$ denote the
Hausdorff semidistance
$${\rm {dist}}(B_1,B_2)=\sup _{\xi_1\in B_1}\inf _{\xi_2\in B_2}{d}(\xi_1,\xi_2), \qquad B_1,B_2\subset\B.$$
A random set ${\A}(\omega )\subset\B$ is said to be a {\it random
attractor} associated with the RDS $\Phi $ if $\P$--a.s.
\begin{enumerate}
\renewcommand{\labelenumi}{\roman{enumi})}
\item 
${\A}(\omega )\mbox{ is compact and, for all } \xi\in\B, \mbox{
the map }\xi\mapsto \mbox{dist}(\xi,{\A}(\omega ))\mbox{ is
measurable,}$
\item  $\Phi (t,\omega ){\A}(\omega )={\A}(\theta _{t}\omega),\:\forall t\geq 0\ \mbox{(invariance)}$, and
\item  for all $D\subset\B$ bounded (and nonrandom)
$\lim _{t\to \infty }\mbox{dist}( \Phi (t,\theta _{-t} \omega)D,
{\A}(\omega ))=0.$
\end{enumerate}
\begin{rem}
Note that $\Phi (t,\theta _{-t} \omega )\xi$ can be interpreted as the position at $t=0$ of the trajectory which was
in $\xi$ at time $-t.$ Thus, the attraction property holds {\it from $t= -\infty .$}
\label{ntime}
\end{rem}
We have the following theorem about existence of random attractors
due to Crauel (\cite{Cr00}, Theorem 3.3):

\begin{thm}
There exists a global random attractor ${\A} (\omega )$ iff
there exists a random compact set $K(\omega )$
attracting every bounded nonrandom set $D\subset\B$.
\label{teorema2}
\end{thm}
Moreover, Crauel \cite{Cr00} proved that random attractors are
unique and, under the ergodicity assumption on $\theta _t$, there
exists a deterministic compact set $K\subset\B$ such that
$\P-a.s.$ the random attractor is the omega-limit set of $K$, that
is,
\[
{\A} (\omega )=\bigcap _{n\geq 0}\overline {\bigcup _{t\geq n}
\Phi   (t,\theta _{-t}\omega )K}.
\]

Our SPDEs solutions are weak semimartingale, which we now define.
\begin{defn}
We call a random field $U(t,x,\omega)$, $x\in G\subset \R^d$, a
weak semimartingale sheet (or simply a weak semimartingale) if
there exists a $p\ge0$ such that the $L^2$ scalar product
$(U(t),\varphi)$ is a semimartingale in time for each fixed
$\varphi\in C_c^p(G)$. If $\A$ is a random attractor corresponding
to a SPDE whose solutions are weak semimartingales, then $\A$ is
called a semimartingale attractor. \label{sa}
\end{defn}

\subsection{Finite dimensional asymptotic behavior}
Here we obtain some results on the finite dimensional asymptotic
behavior of trajectories associated to a random dynamical system,
which we will apply below to the solutions for Allen-Cahn type
SPDEs in (\ref{AC}).

\subsubsection{The random squeezing property}
Suppose the existence of a random compact absorbing set $K(\ome )$
such  that, for some random variable $r(\ome )$, we have that
$\P$-a.s. $K(\ome )\subset B(0,r(\ome )).$ Moreover, suppose that
the $r(\omega)$ is a \textit{tempered} random variable, that is,
$\P$-a.s.
$$
\lim _{t\rightarrow +\infty }\dis \frac{r(\theta _t \ome
)}{e^{\eps t}}=0,
$$
for all $\eps >0.$
\par
Let $\bP:\B\to \bP \B$ be a finite-dimensional orthogonal
projector and let $\bQ=\bI-\bP$ be its counterpart. In what follows, the
main hypothesis (H) is the following:
\newline
Suppose there exist $0<\delta <1$ and a random variable $c(\ome)$
with finite expectation,

\begin{equation}
\EP(c(\ome ))<\ln(1/\delta), \label{50}
\end{equation}
such that, for $\tau \in \R$
\begin{equation}
|\bQ(\Phi (1,\teta _{\tau}\ome )u-\Phi (1,\teta _{\tau}\ome
)v)|\leq \delta \exp\left(\int _\tau ^{\tau +1}c(\teta _s \ome
)ds\right)|u-v|,
 \label{51}
\end{equation}
for all $u,v\in K(\teta _{\tau}\ome ),$ where $\ba \cdot \ba $
denotes the norm in $\B$.
\par
This property is called \textit{the random squeezing property} (RSP)
in Flandoli and Langa \cite{Fl99}, and it was first used in
Debussche \cite{De97} to prove that the random attractor
associated to a RDS has finite Hausdorff dimension $\P$-a.s.

\begin{prop} (\cite{De97}, \cite{Fl99})

Suppose that \eqnref{50}, \eqnref{51} hold. Then, $\P$-a.s.
$$
d_f(\A (\ome ))<+\infty,
$$
where
$$
d_f(K)\doteq \dis \limsup _{\eps \rightarrow 0}\frac{\log N_{\eps
}(K)}{\log (1/\eps )}
$$
denotes the fractal dimension of a compact set $K\subset\B ,$
where $N_{\eps }(K)$ is the minimum number of balls of radius
$\eps $ necessary to cover $K.$ \label{L}
\end{prop}
\subsubsection{Forward and Pullback determining modes}
The following theorem shows the dependence of the asymptotic
behavior, starting with two initial data, on a finite number of degrees of
freedom (Langa \cite{La03}, Theorem 2, and Flandoli and Langa,
Theorem 2):
\begin{thm}
Suppose \eqnref{50} and \eqnref{51} hold. Then,
\begin{enumerate}

\item[a) ] (Forward determining modes)\newline
 given $u_0,v_0\in \B$, suppose that for some $\alpha \geq
0$, we have, $\P$-a.s.,  that
$$
\dis \lim _{t\rightarrow +\infty} \ba \bP(\Phi (t, \omega)u_0-\Phi
(t, \omega) v_0  )\ba \leq \alpha.
$$
Then, $$ \dis \lim _{t\rightarrow +\infty} \ba \Phi (t,
\omega)u_0-\Phi (t, \omega) v_0 \ba \leq \alpha.
$$

\item[b)] (Pullback determining modes)\newline
On the other hand, if $t\in \R$ and for all $r\leq t$, $\P$-a.s.,
and $u_0,v_0\in \B$
$$
\dis \lim _{s\rightarrow +\infty} \ba \bP(\Phie -\Phiev )\ba \leq
\alpha ,
$$
then, for all $r\leq t$,
$$
\dis \lim _{s\rightarrow +\infty} \ba \Phie -\Phiev \ba \leq
\alpha.
$$
\end{enumerate}
\label{DM1}
\end{thm}

\begin{rm} Note that in b) we need a convergence in all final
times $r\leq t$ to get the result. In the next result we will
write a weaker hypothesis for this result.
\end{rm}
\begin{rem}
\begin{itemize}
\item[a)]
If $\alpha =0$ we would have a classical determining modes result
 (cf. Foias and Prodi \cite{Fo67}).
\item[b)]
Due to the fact that the pullback convergence to the random
attractor implies the forward convergence to this set in
probability (Crauel and Flandoli \cite{Cr94}), i.e., for all $\eps
>0$
$$
\dis \lim _{t\to +\infty } \P(\ome \in \Omega : dist (\Phi (t,\ome
)D, \at (\teta _t \ome ))>\eps )=0,
$$
we get that our hypotheses in the previous theorem implies those
in Chueshov et al. \cite{Chu03}, Theorem 2.3, so that the
assertion there also holds in our case.
\end{itemize}
\end{rem}
Using Proposition 2 in Langa \cite{La03}, we also get the pullback
convergence in the previous theorem under a weaker condition.
\begin{thm} (\cite{La03})
Suppose that $u(\ome ), v(\ome )$ are two random variables on the
attractor $\atw $ such that $\P$--a.s.
\begin{eqnarray*}
\Phi (t,\ome )u(\ome)\neq  \Phi (t,\ome )v(\ome),\: \mbox{ for
all } \: t\in \Rp  \mbox{and}\\
 \dis \lim _{s\rightarrow
+\infty} \ba \bP_0(\Phi (t+s ,\teta _{-s}\ome )u(\teta _{-s}\ome )
- \Phi (t+s ,\teta _{-s}\ome ) v(\teta _{-s}\ome ) )\ba =  0
\end{eqnarray*}

whenever  $u(\ome)\neq v(\ome),$ $\P$--a.s.~$($where $\bP_0$ is a
projection which is injective between $\bigcup _{t\in \R}\at
(\teta _t \ome)$ and its image, see Langa and Robinson
\cite{LaRo03} for the existence of such (dense) set of
projections$)$. Then, for all $r\leq t$ we have that
$$
\dis \lim _{s\rightarrow +\infty} \ba \bP_0(\Phi (r+s ,\teta
_{-s}\ome )u(\teta _{-s}\ome ) - \Phi (r+s ,\teta _{-s}\ome )
v(\teta _{-s}\ome ) )\ba =0,\qquad   \P\mbox{--a.s.}
$$
\label{DM2}
\end{thm}

\section{Generalized Allen-Cahn SPDEs and Random PDEs}
\subsection{Definitions}
In this part we consider the SPDE

\begin{equation}
\left\{
\begin{array}{ll}
\displaystyle\frac{\partial U}{\partial t}= \Delta
_{x}U+f(U)+\displaystyle\frac{\partial^2 W}{\partial t\partial x},
& (t,x)\in{\cal O}_L \doteq (0,+\infty )\times (0,L);

\\ U(t,0)=U(t,L)=0, & 0<t<\infty;
\\ U(0,x)=u_0(x), & 0<x<L;
\end{array}
\right. \label{AC}
\end{equation}

where $W(t,x)$ is the Brownian sheet corresponding to the driving
space-time white noise,  written formally as $\partial^2W/\partial
t\partial x$. As noted earlier, we treat white noise as a
continuous orthogonal martingale measure, which we denote by
${\cal W}$. The drift $f:\R\to\R$ is of the form:
\begin{equation}
f(u)=\sum _{k=0}^{2p-1}a_ku^k,\mbox{ with } p\in\N, \mbox{ and
}a_{2p-1}<0, \label{fcond}
\end{equation}
It is not difficult to prove the following elementary inequalities for $f$ (see Temam \cite{T}),
which we use in the proof of \thmref{AC0}:

\begin{enumerate}
\item[i)] There exists $K >0$ such that $f'(v)\leq K,
\:\: \forall \, v\in \mathbb{R}$.
\item[ii)] There exist $c_1 ,c_0 >0$ such that $f(v)v\leq -c_1  v^{2p}+c_0 ,
\:\: \forall \, v\in \mathbb{R}$.
\item[iii)] There exist $k_1,k_0 >0$ such that $|f(v)|\leq k_1 |v|^{2p-1}+k_0 ,
\:\: \forall \, v\in \mathbb{R}$.
\end{enumerate}

We denote the SPDE \eqnref{AC} by $\eh$. We collect here
definitions and conventions that are used throughout this article
(see Walsh \cite{WA} for a whole setting of this type of SPDEs;
see also Allouba \cite{A100,A298}). Filtrations are assumed to
satisfy the usual conditions (completeness and right continuity),
and any probability space $\filspace$ with such a filtration is
termed a usual probability space.
\begin{defn}[Strong and Weak Solutions to $\eh$]
We say that the pair $(U,\Wm)$ defined on the usual probability
space $\filspace$ is a continuous or $L^2$-valued solution to the
stochastic PDE $\eh$ if $\Wm$ is a space-time white noise on
${\cal C}_L\doteq \mathbb{R}_+\times [0,L];$ the random field
$U(t,x)$ is ${\cal F}_t$-adapted $(U(t,\cdot)\in{\cal F}_t\
\forall t)$, with either $U\in C({\cal C}_L;\R)$ $(\mbox{a
continuous solution})$ or $U\in C(\Rp;L^2(0,L))$ $($an
$L^2$-valued solution$);$ and the pair $(U,\Wm)$ satisfies either
one of the following two formulations:
\begin{enumerate}
\item[(TFF)] the test function formulation
\begin{eqnarray*}
(U(t)-u_0,\vfi  )-\int_0^t(U(s),\vfi  '')ds=\int_0^t(f(U(s),\vfi  )ds\\
+\int_0^L\int_0^t \vfix \Wm(ds,dx);\ 0\leq t<\infty, \mbox{ a.s.}\
\P,  \label{TFF}
\end{eqnarray*}
for every $\vfi\in\Theta _{0}^{L}\doteq\left\{ \vfi   \in
C^\infty(\R;\R): \varphi  (0)=\varphi   (L)=0\right\},$ where
$(\cdot,\cdot)$ is the $L^2$ inner product on $[0,L],$ or
\item[(GFF)] the Green function formulation
\begin{eqnarray*}
U(t,x)=\int_0^L\int_0^tf(U(s,y))G_{t-s}(x,y)dsdy+\int_0^L\int_0^tG_{t-s}(x,y)\Wm(ds,dy)\\
+\int_0^LG_t(x,y)u_0(y)dy; \ 0\leq t <\infty   \mbox{ a.s.}\ \P,
 \label{GFF}
\end{eqnarray*}
where $G_{t}(x,y)$ is the fundamental solution to the
deterministic heat equation $(u_t=u_{xx})$ with vanishing boundary
conditions.
\end{enumerate}
A solution is said to be strong if the white noise
$\Wm$ and  the usual probability space $\filspace$ are fixed a
priori and $\Ft$ is the augmentation of the natural filtration for
$\Wm$ under $\P$.  It is termed a weak solution if we are allowed
to choose the usual probability space and the white noise $\Wm$ on
it, without requiring that the filtration be the augmented natural
filtration of $\Wm$.   We say pathwise uniqueness holds for $\eh$
if whenever $(U^{(1)},\Wm)$ and $(U^{(2)},\Wm)$ are two solutions
to $\eh$ on the same probability space $\filspace$, and with respect to the
same white noise $\Wm$, then
$\P\left[U^{(1)}(t,x)=U^{(2)}(t,x);0\le t<\infty, x\in[0,L]\right]=1$. \label{heatsol}
\end{defn}
We often simply say that $U$ solves $\eh$ $($weakly or strongly$)$
to mean the same thing as above.
\begin{rem}
As it is well known $($Walsh \cite{WA}$)$, if the drift and
diffusion coefficients are locally bounded random fields $($in our
case they trivially are for continuous solutions since the
diffusion coefficient $a\equiv1$ and the drift $f$ is clearly
locally Lipschitz under our conditions in \eqnref{fcond}, then
the two formulations (GFF) and (TFF) are equivalent. \label{GTFF}
\end{rem}
\setcounter{section}{3}
\subsection{Existence and uniqueness of solutions}
Let $\beta\ge0$; let $Z_\beta(t,x)$ be the pathwise-unique strong
solution to \eqnref{AC} with $f(Z_\beta)=-\beta Z_\beta$ and
$Z_\beta(0,x)\equiv0$, which is H\"older continuous with
$\alpha_{\mbox{time}}=1/4-\epsilon$ in time and
$\alpha_{\mbox{space}}=1/2-\epsilon$ in space, $\forall\epsilon>0$
(a standard result as in \cite{WA} pp.~321-322).  Let
$V_\beta=U-Z_\beta$, for any solution $U$ to \eqnref{AC}. We see
then that $V_\beta$ satisfies
\begin{eqnarray*}
V_\beta(t,x)=U(t,x)-\int_0^L\int_0^tG_{t-s}(x,y)\left[\Wm(ds,dy)-\beta Z_\beta(s,y)dsdy\right]\\
=\int_0^LG_t(x,y)u_0(y)dy+
\int_0^L\int_0^t\left[f\left(V_\beta+Z_\beta(s,y)\right)
+\beta Z_\beta(s,y)\right]G_{t-s}(x,y)dsdy\\
\doteq \int_0^LG_t(x,y)u_0(y)dy+I_\beta(t,x)=M(t,x)+I_\beta(t,x).
\label{GFFV}
\end{eqnarray*}
That is, $V_{\beta}$ solves the random PDE:
\begin{equation}
\left\{
\begin{array}{l}
 \displaystyle\frac{\partial V_\beta}{\partial t}
=\oA_{x}V_\beta+f\left(V_\beta+Z_\beta\right)+\beta Z_\beta,
  (t,x)\in{\cal O}_L;\\
V_\beta(t,0)=V_\beta(t,L)=0, 0<t<\infty;\\
 V_\beta(0,x)=u_0(x),  x\in[0,L].
 \end{array}
  \right.
\label{RPDE}
\end{equation}
Our first result gives detailed existence, uniqueness, and
comparative regularity results of our SPDE $\eh$ in \eqnref{AC}
and the associated random PDE $\eqnref{RPDE}$.
\begin{thm}
Suppose $f$ satisfies \eqnref{fcond}.
\begin{enumerate}\renewcommand{\labelenumi}{$($\roman{enumi}\/$)$}
\item$($Strong Regularity\/$)$  If $u_0:[0,L]\to\R$ is
Lipschitz continuous and deterministic.    Then, the SPDE $\eh$
has a strong, pathwise-unique, a.s.~$\alpha$-H\"older continuous
solution  with $\alpha_t=1/4-\epsilon$ in time and
$\alpha_x=1/2-\epsilon$ in space, for all $\epsilon>0$.   On the
other hand, under the same conditions on $u_0$, the random PDE
$\eqnref{RPDE}$ has an a.s.~$C^{1,2}((0,\infty)\times(0,L);\R)$
unique solution.
\item$($Weak Regularity\/$)$  For all $0\le s<T$, we have:
\begin{enumerate}
\item[a)]if $u_0\in L^2(0,L)$, there exist a.s.\~unique solutions $U$ and $V$ to $\eh$ and $\eqnref{RPDE}$,
respectively, such that
$$V\in C([s,\infty);L^2(0,L))\cap L^2(s,T;H_0^1(0,L))\cap
L^{2p}(s,T;L^{2p}(0,L)),$$ 
and
$$U\in C([s,\infty);L^2(0,L));$$
\item[b)] if $u_0\in H_0^1(0,L)$, then the a.s.~unique solutions $U$ and $V$ are such that
$$V\in C([s,\infty);H_0^1(0,L))\cap L^2(s,T;H^2(0,L))\cap
L^{2p}(s,T;L^{2p}(0,L)),$$
and
$$U\in C([s,\infty);L^2(0,L)).$$
\end{enumerate}
and hence, $V\in C([s+\varepsilon,\infty);H_0^1(0,L)) \cap
L^2(s+\varepsilon ,\infty ;H^2(0,L)),$ for every $u_0\in L^2(0,L)$ and
$\varepsilon>0.$
\end{enumerate}

 \label{AC00}
\end{thm}
\begin{rem}
In addition to the existence, uniqueness, and regularity for the
SPDE $\eh$, our proof of \thmref{AC00} gives detailed strong, as
well as weak, regularity results for the random PDE \eqnref{RPDE}
associated with our SPDE \eqnref{AC}.   The strong regularity
results are for completeness, and they are not needed for the rest
of the paper.  Two points are worth emphasizing: 1. solutions to
the random PDE \eqnref{RPDE} are typically much smoother than
solutions to the Allen-Cahn SPDE $\eh$ and 2. while increasing the
regularity of the initial function $u_0$ has a considerable effect
on smoothing out the random PDE solution $($if $u_0$ is Lipschitz
then the solution $V$ is in $C^{1,2}((0,\infty)\times(0,L));$ the
most regularity we can guarantee for the Allen-Cahn SPDE solution
is H\"older continuity (with H\"older exponents $1/4$ in time and
$1/2$ in space) regardless of how smooth the initial data is. This
of course is a direct result of the fact that the driving noise is
white in both space and time.  In the case of a time only white
noise, the solution $U$ to the Allen-Cahn SPDE driven by such
noises is spacially much smoother than our solutions (typically at
least in $H^1(0,L)$, e.g., see \cite{Cr94}).  For more on the
effects of our rougher noise on the proof of the existence of the
attractor see \remref{sae} below. \label{reg}
\end{rem}

\begin{proof}(of \thmref{AC00})
We note that when $p=1$ in \eqnref{fcond} $f$ is Lipschitz and the
strong existence, pathwise uniqueness, and H\"older regularity for
$\eh$ follow from standard results (see \cite{WA}).
\par
We now turn to the case $p>1$.  For simplicity and without loss of
generality, we assume $\beta=0$.  Let $Z\doteq Z_0 $ and $V\doteq
V_0$.  Clearly, the existence and uniqueness for $\eh$ is
equivalent to the existence and uniqueness for the corresponding
random PDE \eqnref{RPDE}.  This is because $Z$ is the
pathwise-unique strong solution (see \cite{WA}) to the standard
heat SPDE and $V+Z$ is a solution to $\eh$ if and only if $V$
solves \eqnref{RPDE}. Furthermore $Z(t,x)$ is
a.s.~$\alpha$-H\"older-continuous with $\alpha_t=1/4-\epsilon$ in
time and $\alpha_x=1/2-\epsilon$ in space, for all $\epsilon>0$
(see \cite{WA}), and it vanishes at $0$ and $L$. For the rest of
the proof, we fix $\omega\in\Omega$, and treat the path-by-path
deterministic version of our random PDE \eqnref{RPDE}. Following
the proof of Theorem 1.1 in Temam \cite{T}, Chapter III---and for
the usual Sobolev spaces $H_0^1(0,L):=\{v\in
H^1(0,L),v(0)=v(L)=0\}$ and $H^2(0,L)$---we have $\P$-a.s.~that
there is a unique continuous (in $(t,x)$) solution $V$ to
\eqnref{RPDE} satisfying \eqnref{GFFV} if $u_0:[0,L]\to\R$ is
deterministic and continuous. This implies that $|f(V+Z)|\le
C_1<\infty$ on $[0,t]\times [0,L]$; thus $I_0(t,\cdot)\in
C^1(0,L)$ with $|DI_0(t,x)|\le C C_1t^\frac12$ (the smoothness for
$I_0$ is obtained throughout as in Theorems 2 to 5 in Chapter 1 of
\cite{F}) and hence $V(t,\cdot)\in C^1(0,L)$ for every $t$ (the
first term in \eqnref{GFFV}, $M$, is in $C^2(0,L)$ whenever $u_0$
is continuous on $[0,L]$).   If additionally $u_0$ is Lipschitz on
$[0,L]$; then $f(V+Z)$ is H\"older continuous on $[0,L]$,
uniformly locally in $t$.  To see this, remember that when $u_0$
is Lipschitz on $[0,L]$ then, with $M$ as defined as in
\eqnref{GFFV}, we have $M\in C^2(0,L)$ and
\begin{equation}
D M(t,x)=\int_0^Lu_0(y)\frac{\partial}{\partial x}G_t(x,y)dy\le K.
\label{70}
\end{equation}
The bound in \eqnref{70} again follows from standard analysis
methods as in Chapter 1 in \cite{F} (see also \lemref{derbd} below
for a probabilistic proof of this fact on $\Rd$, $d\ge1$).

The bound in \eqnref{70} and the bound that we have for
$DI_0(t,x)$ imply that $V$, and hence $f(V+Z)$, is H\"older
continuous on $[0,L]$, uniformly locally in $t$.  This, in turns
implies that $I_0(t,\cdot)\in C^2(0,L)$ and hence $V(t,\cdot)\in
C^2(0,L)$ for every $t$. The temporal regularity for $V$ is proved
similarly and we omit it, and we obtain that $V\in
C^{1,2}((0,\infty)\times(0,L);\R)$. It is then clear that
$U(t,x)=V(t,x)+Z(t,x)$ is the pathwise-unique (because uniqueness
holds a.s.~for both $V$ and $Z$) strong solution (because the
white noise $\Wm$ is fixed throughout) of \eqnref{AC}, and that
$U$ is $\P$ a.s.~H\"older continuous under our conditions on $u_0$
with $\alpha_t=1/4-\epsilon$ in time and $\alpha_x=1/2-\epsilon$
in space, for all $\epsilon>0$ (since both $V$ and $Z$ are)
\par
In addition, we also get $\P$-a.s.~that for all $0\le s<T$:
\begin{enumerate}
\item[a)]if $u_0\in L^2(0,L)$, there exists a unique
solution $$V\in C([s,\infty);L^2(0,L))\cap L^2(s,T;H_0^1(0,L))\cap
L^{2p}(s,T;L^{2p}(0,L)),$$ 
\item[b)] if $u_0\in H_0^1(0,L)$, then there exists a unique solution
$$V\in C([s,\infty);H_0^1(0,L))\cap L^2(s,T;H^2(0,L))\cap
L^{2p}(s,T;L^{2p}(0,L)),$$
\end{enumerate}
and hence, $V\in C([s+\varepsilon,\infty);H_0^1(0,L)) \cap
L^2(s+\varepsilon ,\infty ;H^2(0,L)),$ for every $u_0\in L^2(0,L)$ and
$\varepsilon>0.$    Again, the assertions about the existence, uniqueness and weak regularity
of $U$ (part ii) a) and b) in \thmref{AC00}) easily follow from the
corresponding results for $V$ (parts a) and b) above), the regularity of $Z$,
and the fact that $U(t,x)=V(t,x)+Z(t,x)$.
\end{proof}

\subsection{Growth rates for $Z_\beta$}
In this subsection, we obtain asymptotic growth rates of interest related to $Z_\beta$.
\begin{lem}  Let $Z_\beta$ be as in the proof of \thmref{AC00}.  Then $Z_\beta$ may be rewritten as
$$Z_\beta(t,x)=\int_0^L\int_0^t G_{\beta,t-s}(x,y)\Wm(ds,dy),$$
where $G_{\beta}$ is the fundamental solution to the noiseless
version of \eqnref{AC} with $f(Z_\beta)=-\beta Z_\beta$, with
Dirichlet boundary conditions (\cite{WA}).  Let
$\hat{Z}_\beta(t)\doteq\sup_{0\le x\le L}Z_\beta(t,x)$.  Then,
\begin{enumerate}
\item[i)] For each $0<p<3$ and
$0\le\gamma<1\wedge(3-p)$, there exists a constant $K>0$ such that
\begin{equation}
\int_0^L\int_0^tG_{\beta,t-s}^p(x,y)dsdy\le K(x\wedge(L-x))^{\gamma}t^{(3-p-\gamma)/2};
\,x\in(0,L),\, t>0, \beta\ge0.
\label{Greenbounds}
\end{equation}
\item[ii)]  $\P[\hat{Z}_\beta(t)>t^{\frac14+\epsilon}]\le Kt^{-\epsilon}\to0$ as $t\to\infty$,
for every $\epsilon>0$ and every $\beta\ge0$
for some universal constant $K>0$.
\item[iii)]  If \begin{equation*}
Z^{\varphi}_{\beta}(t)\doteq(Z_\beta(t),\varphi
)-\int_0^t(Z_\beta(s),\varphi  '')ds +\int_0^t(\beta
Z_\beta(s),\varphi  )ds;\quad 0\le t<\infty,
\end{equation*} then, for every $\beta\ge0$
and $\varphi\in\Theta _{0}^{L}$, $Z^{\varphi}_{\beta}(t)/t\to0$ as
$t\to\infty$ $\P$-a.s.
\end{enumerate}
\label{asympt}
\end{lem}

\begin{proof}  \begin{enumerate}
\item[i)]  The Green function,  $G_\beta$ is easily found to be
$$G_{\beta,t}(x,y)=\frac{e^{-\beta t}}{\sqrt{4\pi t}}\sum_{n=-\infty}^{\infty}
\left\{\exp\left(-\frac{(2nL+y-x)^2}{4t}\right)-
\exp\left(-\frac{(2nL+y+x)^2}{4t}\right)\right\}.$$ It is clearly
enough to prove the estimate on $G_\beta$ for $\beta=0$, and we
will denote $G_0$ by simply $G$.  Now, let $B^x$ be the scaled
Brownian motion $B^x=\sqrt2\tilde{B}^{x/\sqrt2}$, starting at some
$x\in\L\doteq(0,L)$, where $\tilde{B}^x$ is a standard Brownian
motion starting at $x$. Let $\tau^x_\L\doteq
\inf\left\{t>0;B_t^x\notin\L\right\}$; then $G_t(x,y)$ is the
density of $B^x$ killed at $\tau^x_\L$.  We easily have
\begin{equation*}
\begin{array}{l}
\mbox{a)}\ G_r(x,y)=\E\left[G_{r/4}(B^x_{r/4},y)1_{[\tau^x_\L\ge r/4]}\right]\\
\mbox{b)}\ \dstyle{\int_0^L G_{r/4}^p(\xi,y)dy\le
K\int_0^L\left[\frac{1}{\sqrt{\pi r}}\,
e^{-(\xi-y)^2/r}\right]^pdy \le\frac{K}{|r|^{(p-1)/2}}} \label{MP}
\end{array}
\end{equation*}
where we simply used the Markov property to obtain \eqnref{MP} a).
Now, applying H\"older inequality to \eqnref{MP} a) and then using
\eqnref{MP} b) we get
\begin{eqnarray*}
\int_0^L G^p_{t-s}(x,y)dy&\le\E\left[\int_0^L\left|G_{(t-s)/4}(B^x_{(t-s)/4},y)
\right|^pdy\,1_{[\tau^x_\L\ge (t-s)/4]}\right]\\
&\le\displaystyle \frac{K\P\left[\tau^x_\L\ge
(t-s)/4\right]}{|t-s|^{(p-1)/2}} \label{bstand}
\end{eqnarray*}
But we also have
\begin{equation*}
\P[\tau^x_\L\ge r]\le\left(\P[\tau^x_\L\ge r]\right)^\gamma\le K\left[\frac{x\wedge(L-x)}{\sqrt r}\right]^\gamma
\label{stand}
\end{equation*}
where the inequalities in \eqnref{stand} follow from the standard
facts: $\P[\tau^x_\L\ge r]\le1$ and $\P[\tau^x_\L\ge r]\le K
x\wedge(L-x)/\sqrt r$.    Finally, \eqnref{bstand} and
\eqnref{stand} give us
\begin{eqnarray*}
&\int_0^t\int_0^L G^p_{t-s}(x,y)dyds
\le\int_0^t\frac{K\P\left[\tau^x_\L\ge (t-s)/4\right]}{(t-s)^{(p-1)/2}}ds\\
&\le K [x\wedge(L-x)]^\gamma\int_0^t
\frac{ds}{(t-s)^{(p+\gamma-1)/2}} =K [x\wedge(L-x)]^\gamma
t^{(3-p-\gamma)/2} \label{estend}
\end{eqnarray*}
\item[ii)] Using \eqnref{Greenbounds} with $p=2$ and $\gamma=0$ along with
Chebyshev and Burkholder inequalities;
we have, for any $\beta\ge0$, that
\begin{eqnarray*}
\P[\left|\hat{Z}_\beta(t)\right|>t^{\frac14+\epsilon}]\le
\frac{\EP|\hat{Z}_\beta(t)|}{t^{\frac14+\epsilon}}\le
K\frac{\EP\left(\dstyle{\int_0^L\int_0^t\sup_{0\le x\le
L}G^2_{\beta,t-s}(x,y)dsdy}\right)^{1/2}}{t^{\frac14+\epsilon}}
\le Kt^{-\epsilon},
\end{eqnarray*}
\item[iii)]  First, note that for every test function $\vfi\in\Theta_0^L$,
\begin{equation*}
Z^{\varphi}_{\beta}(t)=\int_0^L\int_0^t\vfi(y)\Wm(ds,dy);
\quad0\le t<\infty, \vfi\in\Theta_0^L. \label{Zfi}
\end{equation*}
Using Doob's maximal inequality, we now have
\begin{eqnarray*}
\P\left[\sup_{2^n\le
t\le2^{n+1}}\frac{\left|Z^{\varphi}_{\beta}(t)\right|}{t}>\epsilon\right]\le
\frac1{\epsilon^2}\EP\left[\sup_{2^n\le
t\le2^{n+1}}\left(\frac{Z^{\varphi}_{\beta}(t)}
{t}\right)^2\right] \\
\le\frac1{2^{2n}}\EP\left[\sup_{2^n\le
t\le2^{n+1}}(Z_\beta^{\varphi})^2\right]\le\frac4{2^{2n}}\EP
(Z_\beta^{\varphi})^2(2^{n+1})\le\frac{8K_\vfi L}{2^{n}};\
\forall\epsilon>0, n\ge1, \label{Doob}
\end{eqnarray*}
where $K_\vfi$ is the bound on $\vfi^2$.  An easy application of
Borel-Cantelli lemma gives us that, for every $\vfi\in\Theta_0^L$,
$$Z^{\varphi}_{\beta}(t)/t\to0 \mbox  { as } t\to\infty,\ \P-\mbox{a.s.~}$$
\end{enumerate}
The proof is complete.
\end{proof}
\subsection{Existence of the semimartingale attractor}
Let $U$ be the solution to $\eh$ on $\filspace .$ For any
functions $J(t,x)$ and $j(x)$ let
\begin{eqnarray}
J^\vfi  (t)\doteq(J(t),\vfi  ) \mbox{ and } j^\vfi  =(j,\vfi);\
\forall \vfi  \in\Theta _{0}^{L}. \label{smartdef}
\end{eqnarray}
Then, by the assumptions on $f$ we easily have that $\{U^\vfi (t);
t\in\Rp\}$ is a semimartingale on $\filspace$ for each $\vfi  $
since \eqnref{TFF} gives $ \P\mbox{--a.s.}$

\begin{eqnarray*}
U^\varphi  (t)=u_0^\vfi  +\int_0^tU^{\vfi  ''}(s)ds+
\int_0^t(f(U(s)),\vfi  )ds+\int_0^L\int_0^t \vfix \Wm(ds,dx);\
0\leq t<\infty.  \label{Uphi}
\end{eqnarray*}

So, by \defnref{sa} the random field $U(t,x,\omega)\doteq
U(t,x,\omega; u_0)$ is a weak semimartigale sheet, starting at
$u_0(x)$. Now, set

\begin{equation}
 \Phi_U(t-s,\omega)U(s,\cdot,\omega)=U(t,\cdot,\omega).
\end{equation}

In particular, we can define in $\B=L^2(0,L)$ the random dynamical
system
\begin{equation}
 \Phi_U(t,\omega)u_0=U(t,\cdot,\omega; u_0).
\end{equation}
As in \defnref{sa}, we call a random attractor $\A$ associated
with ${\Phi_{U}}$ a semimartingale attractor (as we mentioned on
p.~2, our noise setting allows us to to treat a related type of
attractors we call semimartingale measure attractor. More on this
in an upcoming article). When we want to emphasize the type of
semimartingales captured by the attractor, we say
weak-semimartingale functional attractor. Our result for $\A$ can
now be stated as
\begin{thm}
Suppose $f$ satisfies \eqnref{fcond}, $u_0\in L^2(0,L)$, and $u_0$
is deterministic.  Then, the SPDE $\eh$ possesses a finite
dimensional semimartingale attractor $\A\subset L^2(0,L)$.
\label{AC0}
\end{thm}
\begin{rem}
As we mentioned in \remref{reg}, the fact that our driving noise
is white in both space and time leads to a much less spatial
regularity for our solutions of $\eh$ as compared to SPDEs driven
by noises that are white only in time: in the case of Allen-Cahn
SPDEs with noises that are white in time only, the solution $U$ is
typically at least in $H^1$ (see e.g., \cite{Cr94}); while in our
case the solution $U$ of $\eh$ is not even in $H^1$ (even if we
start with a $C^\infty$ initial function $u_0$). So, the proofs in
\cite{Cr94}, for example, use the fact that for their time only
white noise one may apply the Laplacian to the noise (not to
mention the solution of the SPDE); and one may use directly the
$H_0^1$ norm on the solution $U$ of the SPDE.  All of these facts
do not apply in our case since neither $U$ nor $Z_\beta$ are even
in $H_0^1$ let alone $H^2$, and we must proceed differently. Thus,
our proof relies heavily on the regularity of the solution
$V_\beta$ to the associated random PDE \eqnref{RPDE}; and the
apriori estimates needed to establish the existence of the
attractor are substantially harder in our case, requiring more
elaborate fundamental inequalities (see the proof below, along
with the modified inequalities in the Appendix, and compare it to
the proof in \cite{Cr94}). Also, adding to the difficulty in our
case is the combination of this lack of spatial regularity and the
order of the nonlinearity in $\eh$, $2p-1$; which makes proving
the existence of the attractor in $L^2$ more difficult than in the
case of Burgers type SPDEs (even those driven by space-time white
noise), whose nonlinearity is effectively second order.
\newline
Finally, if $u_0$ is assumed to be Lipschitz; then, by
\thmref{AC00}, the derivatives in $\Delta V$ below exist in the
strong classical sense; and solutions to $\eqnref{RPDE}$ are
classical. \label{sae}
\end{rem}
\begin{proof}(of \thmref{AC0})     In light of \remref{ntime}, we look at our
white noise $\Wm$ as a two-sided (in time) space-time white noise
on $\pspace$.   I.e., if $$\Omega=\left\{\omega\in C(\R\times[0,L];\R):
\omega(0,x)=\omega(t,0)=0\right\}$$ with $\P$ being the product measure of two
Brownian-sheet measures on the negative and positive time parts of $\Omega$; then
$W(t,x)=\omega(t,x)$ and $\Wm$ is the white noise corresponding to
$W$.   We accordingly extend the time domain of the $Z_\beta$ and $V_\beta$
to negative time as well, in the obvious standard way,
and we also refer to the extended $Z_\beta$ and $V_\beta$ as $Z_\beta$ and
$V_\beta$, respectively.

It can easily be checked that $\Phi_U$ satisfies properties (i)
and (ii).  Let $\beta>0$; then, as in the proof of \thmref{AC00},
\eqnref{RPDE} has a unique solution $V_\beta$ with the same
regularity as $V$ for all $-\infty<s<T$. Multiplying \eqnref{RPDE}
by $V_\beta^{2p-1}$ and integrating over space $[0,L]$ in
\eqnref{RPDE} (with the standard convention $(\Delta
V_\beta,V_\beta^{2p-1})=-(DV_\beta,DV_\beta^{2p-1})$ for
$V_\beta\in H_0^1(0,L)$;  e.g., \cite{Ro01}); using Young's and
H\"older's inequalities repeatedly and the generalized Poincar\'e
inequality $|v|_{L^p(0,L)}\le L|Dv|_{L^p(0,L)},\ p\ge1$ along with
its consequence
$-(DV,DV^{2p-1})\le-((2p-1)/p^2L)|v|_{L^{2p}(0,L)}^{2p}$ (see
\lemref{gPoincare} and \lemref{inprd}); and using elementary
inequalities on $f$ and elementary manipulations, we get after
collecting terms (with $L^{2p}=L^{2p}(0,L)$) that
\begin{eqnarray*}
\frac1{2p}\frac{d}{dt}|V_\beta(t)|_{L^{2p}}^{2p}
 \le-\frac{2p-1}
{p^2L}|V_\beta|_{L^{2p}}^{2p} \end{eqnarray*}
\begin{eqnarray*}
+|V_\beta+Z_\beta|_{L^{4p-2}}^{4p-2} \left\{\left[\sum_{{1\le i\le
2p-1 \atop i \mbox{ odd}}}\left( {2p-1 \atop
i}\right)K_{1,i}\frac{4p-i-2}{4p-2}\epsilon_i\right]-c_{1,0}
\right\}\\+\sum_{{1\le i\le 2p-1 \atop i \mbox{  odd}}}\left(
{2p-1 \atop
i}\right)\left[K_{0,i}L^{\frac{4p-i-2}{4p-2}}|Z_\beta(t)|_{L^{4p-2}}^i
+\frac{i K_{1,i}}{(4p-2)\epsilon_i^\frac{4p-i-2}{i}}|Z_\beta(t)|_{L^{4p-2}}^{4p-2}\right] \\
+\sum_{{0\le i\le 2p-2 \atop i \mbox{  even}}}\left( {2p-1 \atop
i}\right)c_{0,i}|Z_\beta(t)|_{L^{i}}^i+\beta\left[\frac{\epsilon(2p-1)}{2p}
|V_\beta|_{L^{2p}}^{2p}+\frac1{2p\epsilon^{2p-1}}|Z_\beta|_{L^{2p}}^{2p}\right]
\label{7}
\end{eqnarray*}
Choosing the Young's $\epsilon_i$'s so that
$$c_{1,0}\ge\sum_{{1\le i\le 2p-1 \atop i \mbox{  odd}}}\left( {2p-1 \atop i}\right)K_{1,i}\frac{4p-i-2}{4p-2}\epsilon_i$$
we get
\begin{eqnarray*}
&\frac{d}{dt}|V_\beta(t)|_{L^{2p}}^{2p}\le|V_\beta|_{L^{2p}}^{2p}\left[\beta
\epsilon(2p-1)-\frac{4p-2}{pL}\right]+2p\sum_{{0\le i\le 2p-2 \atop i
\mbox{ even}}}\left( {2p-1 \atop i}\right)c_{0,i}|Z_\beta(t)|_{L^{i}}^i\\
\\&+2p\left\{\sum_{{1\le i\le 2p-1 \atop i \mbox{
odd}}}\left( {2p-1 \atop i}\right)\left[K_{0,i}L^{\frac{4p-i-2}{4p-2}}|Z_\beta(t)|_{L^{4p-2}}^i
+\frac{i
K_{1,i}}{(4p-2)\epsilon_i^\frac{4p-i-2}{i}}|Z_\beta(t)|_{L^{4p-2}}^{4p-2}
\right]\right\}\\&+\frac{\beta}{\epsilon^{2p-1}}|Z_\beta(t)|_{L^{2p}}^{2p}
\label{14}
\end{eqnarray*}
Choosing $\epsilon$ such that
$$-\lambda=\left[\beta\epsilon(2p-1)-\frac{4p-2}{pL}\right],$$
where $\lambda=\lambda_1/2$ ($\lambda_1$ is the first positive
eigenvalue for the Laplace operator), we get
\begin{eqnarray*}
\frac{d}{dt}|V_\beta(t)|_{L^{2p}}^{2p}+\lambda
|V_\beta(t)|_{L^{2p}}^{2p}
\le\frac{\beta}{\epsilon^{2p-1}}|Z_\beta(t)|_{L^{2p}}^{2p}+2p\sum_{{0\le
i\le 2p-2 \atop i \mbox{  even}}} \left( {2p-1 \atop
i}\right)c_{0,i}|Z_\beta(t)|_{L^{i}}^i
\end{eqnarray*}
\begin{eqnarray*}
 +2p\left\{\sum_{{1\le i\le
2p-1 \atop i \mbox{  odd}}}\left( {2p-1 \atop
i}\right)\left[K_{0,i}L^{\frac{4p-i-2}{4p-2}}|Z_\beta(t)|_{L^{4p-2}}^i
+\frac{i
K_{1,i}}{(4p-2)\epsilon_i^\frac{4p-i-2}{i}}|Z_\beta(t)|_{L^{4p-2}}^{4p-2}\right]\right\},
\label{040}
\end{eqnarray*}
where $|Z_\beta(t)|_{L^{0}}^0\doteq1$.  Now, picking $\beta$ large
enough  (similarly to \cite{CDF}), using Gronwall's Lemma, and
letting $P(t)$ denote the term on the right hand side of the above
inequality; we can deduce that there exists an $s_1(\omega)$ such
that if $s<s_1(\omega)$ and $-1\le t\le0$,
\begin{eqnarray*}
|V_\beta(t)|^{2p}_{L^{2p}}
&\leq\kappa\left[|V_\beta(s)|^{2p}_{L^{2p}} e^{\lambda s}+\int
_{-\infty }^{0}P(r)e^{\lambda r }dr\right]\le
r_0(\omega)=1+\kappa\int _{-\infty }^{0}P(r)e^{\lambda r }dr.
\label{stom1}
\end{eqnarray*}
The at most polynomial growth of the norms of $Z_\beta(t)$ in
$P(t)$ in \eqnref{040} as $t\to-\infty$ (and hence the finiteness
of $r_0=r_0(\omega)$) follows straightforwardly from standard
results concerning the elementary SPDE corresponding to $Z_\beta$
(e.g., see \cite{TDD} Lemma 4.1 and the ensuing discussion as well
as \cite{CDF,WA}; see  also \lemref{asympt} here). On the other
hand, if $\bA=-\Delta_x$ (the negative Dirichlet Laplacian), then
$e^{-t\bA}v(x)=\int_0^Lv(y)G_t(x,y)dy$, and so using \eqnref{GFFV}
on the interval $[-1,0]$ and applying the operator $\bA^{1/8}$
gives us
\begin{equation}
\begin{array}{l}
\left|\bA^{\frac18}V_\beta(0)\right|_{L^2}\le\left|\bA^{\frac18}e^{-\bA}
V_\beta(-1)\right|_{L^2}\\+\int_{-1}^0\left\{\left|\bA^{\frac18}e^{\bA
s}
f(V_\beta(s)+Z_\beta(s))\right|_{L^2}+\beta\left|\bA^{\frac18}e^{\bA
s} Z_\beta(s)\right|_{L^2}\right\}ds \end{array} \label{050}
\end{equation}
To go further, we use the following Sobolev embedding and smoothing properties
of the semigroup $(e^{-{t\bA} })_{t\ge0}\mbox{ and }e^{-\bA}$:
\begin{equation*}
|z|_{L^2} \le C_2|z|_{W^{\frac12,1}}\mbox{ $\forall z\in
W^{\frac12,1}(0,L)$. }
\end{equation*}
\begin{equation}
|e^{-\bA t}z|_{W^{s_2,r}}
\le C_1\left(t^{\frac{s_1-s_2}{2}}+1\right)|z|_{W^{s_1,r}}
\mbox{ $\forall z\in W^{s_1,r}(0,L)$, }-\infty<s_1\le s_2<\infty, r\ge1\\
\label{smth}
\end{equation}
\begin{equation*}
\left|\bA^\frac18e^{-\bA}\right|_{{\cal L}(L^2(0,L))} \le C_0.
\end{equation*}

Now, using \eqnref{smth} with $r=1,\ s_1=-1/4,\ s_2=1/2$ we see
that
\begin{eqnarray*}
\left|\bA^{\frac18}e^{\bA
s}f(V_\beta(s)+Z_\beta(s))\right|_{L^2}\\\le C_1 C_2
\left(t^{-\frac38}+1\right)\left|\bA^{\frac18}
f(V_\beta(s)+Z_\beta(s))\right|_{W^{-\frac{1}4,1}}\\
\le C_1 C_2 \left(t^{-\frac38}+1\right)\left|f(V_\beta(s)+Z_\beta(s))\right|_{L^{1}} \\
\le C_1 C_2 \left(t^{-\frac38}+1\right)\left[k_1|V_\beta+Z_\beta|_{L^{2p-1}}^{2p-1}+k_0L\right] \\
\le C \left(t^{-\frac38}+1\right)\left[|V_\beta|_{L^{2p-1}}^{2p-1}+|Z_\beta|_{L^{2p-1}}^{2p-1}+1\right] \\
\le \kappa \left(t^{-\frac38}+1\right)
\left[|V_\beta|_{L^{2p}}^{2p-1}
+|Z_\beta|_{L^{2p-1}}^{2p-1}+1\right]
\end{eqnarray*}
\begin{eqnarray}
\le   \kappa
\left(t^{-\frac38}+1\right)\left[r_0^{\frac{2p-1}{2p}}+|Z_\beta|_{L^{2p-1}}^{2p-1}+1\right],
\label{21}
\end{eqnarray}
where $\kappa$ depends on $L$ and $p$.  Using \eqnref{smth},
\eqnref{21}, and the fact that $|V_\beta|_{L^2}\le
L^{(p-1)/2p}|V_\beta|_{L^{2p}}$ we arrive at
\begin{eqnarray*}
\left|\bA^{\frac18}V_\beta(0)\right|_{L^2} \le
R_0(\omega)=\kappa_0r_0^{\frac{1}{2p}}+\kappa_1\int_{-1}^0\left(|s|^{-\frac38}+1\right)
\left[r_0^{\frac{2p-1}{2p}}+|Z_\beta(s)|^{2p-1}_{L^{2p-1}}+1\right]\\+
\beta \left|\bA^{\frac18}e^{\bA s}Z_\beta(s)\right|_{L^2}ds,
\label{28}
\end{eqnarray*}
where the constants $\kappa_0,\kappa_1$ depend on $L$ and $p$.

Lastly, if we let $K(\omega)$ be the ball in ${\cal
D}(\bA^\frac18)$ of radius
$R_0(\omega)+|\bA^\frac18Z_\beta(0,\omega)|_{L^2}$; then $
K(\omega)$ is compact because $\bA$ has a compact inverse, and it
is obviously an attracting set at time $0$. The existence of the
attractor follows.
\par
To prove the finite dimensionality of the semimartingale
attractor, suppose $f$ satisfies \eqnref{fcond} and $u_0$ is
Lipschitz continuous and deterministic.  First, observe that
\eqnref{51} is a consequence of the driving space-time white
noise being additive. Indeed, for any two solutions
$U^{(1)},U^{(2)}$ of
$$
 \displaystyle\frac{\partial U}{\partial t}=
 \Delta _{x}U+f(U)+\displaystyle\frac{\partial^2 W}{\partial t\partial x}
$$
with respect to the same white noise $\Wm$ (this can always be assured since
our solutions are strong) and with corresponding initial data $u_0(x),v_0(x),$
we have that the difference $Y(t)=U^{(1)}-U^{(2)}$ satisfies
\begin{equation}
 \displaystyle\frac{\partial Y}{\partial t}=
 \Delta _{x}Y+f(U^{(1)})-f(U^{(2)}).
\label{77}
\end{equation}
I.e., the space-time white noise no more explicitly drives
\eqnref{77}. \par We then can follow exactly the computations in
Debussche \cite{De97}, Section 3.1 (see also Flandoli and Langa
\cite{Fl99}) to verify \eqnref{51}. Indeed, if we call $z_i=\bQ
\Phi (t,\omega )u_0,\: i=1,2,$ then we have, for $z=z_1-z_2, $
$$
\dis \frac{dz}{dt}+Az=\bQ (f(U^{(1)})-f(U^{(2)})),
$$
and then, as in Debussche \cite{De97}, we can write
$$
\dis \frac{d}{dt}|z|^2+\lambda _{m+1}|z|^2\leq
|f(U^{(1)})-f(U^{(2)})|^2_{L^{6/5}},
$$
and, for $m_p=4(p-1),$
$$
|f(U^{(1)})-f(U^{(2)})|_{L^{6/5}}^{2}\leq c(|U^1|^2_{L^{6(p-1)}}+
|U^2|^2_{L^{6(p-1)}})^{m_p}|U^1-U^2|^2.
$$
But note that we have obtained an absorbing radius $r(\omega )$
for $|U^i|^2_{L^{6(p-1)}}, \: i=1,2,$ so that
$$
\dis \frac{d|z|^2}{dt}+\lambda _{m+1}|z|^2\leq C\, r(\omega
)^{m_p}|U^1-U^2|^2,
$$
which leads straightforwardly to the squeezing property by Gronwall
Lemma for $m$ big enough. That $r(\omega) $ is tempered is a
consequence of the at most polynomial growth of this random
variable.
\par On the other hand, \eqnref{51} is also true by Da Prato and
Zabczyk (\cite{Da92}, p. 336; see also Crauel et al. \cite{CDF},
Section 3.2). \par Thus, we can conclude that the semimartingale
attractor of the SPDE \eqnref{AC} has $\P$-a.s.~finite fractal
dimension. This follows as in Langa \cite{La03}, Proposition 3;
which generalizes to the stochastic case Lemma 2.2 in Eden et al.
\cite{Ed94} (see also Robinson \cite{Ro01}).
\end{proof}
Also, note that we immediately have, because of the RSP, the
determining modes result
\begin{thm}
The SPDE \eqnref{AC} satisfies a (forward and pullback)
determining modes result as in Theorems \ref{DM1}, \ref{DM2}.
\end{thm}

\section{Comments and Conclusions}
This article is another step in our work--started in \cite{AL0}
and which is being continued in different directions in
\cite{ALII,AL2,AL3}---of studying the asymptotic behavior of
different types of SPDEs driven by space-time white noise.   Here,
we build on the theory of random attractors; and we generalize it
to our space-time continuous orthogonal local martingale measure
noise setting.  In so doing, we characterize the randomness of our
attractor as one coming from semimartingale-type solutions.  We
believe this characterization is a key step that allows us to use
stochastic analytical tools to gain a deeper understanding of the
stochastic aspects of these random attractors; and we are hopeful
it will point out more clearly the differences between the
attractors associated with SPDEs and those associated with their
non-random counterparts. One consequence of this characterization
would lead to the notions of semimartingale decomposition of the
random attractor, and that of semimartingale measure attractor,
based on the notion of semimartingale measures, which generalizes
the concept of continuous orthogonal semimartingale measures
introduced in Allouba \cite{A298} and it is different from the
measure which is the law of solutions.

We focus in this article on the stochastic Allen-Cahn equations, driven by space-time
white noise; and we give a thorough treatment of the semimartingale functional attractor
in this case.   In particular, the existence of a finite fractal dimension semimartingale
attractor and some results on determining modes have been proved.

\section*{Acknowledgements}  We would like to sincerely thank
Roger Temam for introducing us to each other and for his
encouragements. We also would like to thank the referee for a
careful reading of our paper which led to a better written clearer
article. The first author has been supported in part by NSA grant
MDA904-02-1-0083, and the second one by M.E.C.~(Spain, Feder),
 Proyecto BFM2002-03068.

\appendix
\setcounter{section}{1}

\section*{Appendix A: Some inequalities}
\setcounter{lem}{0} The first lemma generalizes Poincar\'e's
inequality to all $L^p(0,L)$, $p\ge1$.
\begin{lem}[$L^p$ Poincar\'e's Inequality]   Suppose $v\in C^1((0,L);\R)$, for some $L>0$,
 with $v(0)=0$; then
$$|v|_{L^p}\le L |Dv|_{L^p}, \mbox{ for all }p\ge1.$$
\label{gPoincare}
\end{lem}
\begin{proof}
We have $v(x)=\int_0^xDv(y)dy$, $0<x\le L$; and so using H\"older's inequality we get
$$|v(x)|\le \left\{\begin{array}{ll} L^{\frac{1}{p'}}|Dv|_{L^p};& p>1\mbox{ and }
p'=\frac{p}{p-1},\\
|Dv|_{L^1}.&\end{array}\right.$$
Consequently,
$$|v|_{L^p}\le \left\{ \begin{array}{ll}\left(\int_0^L L^{\frac{p}{p'}}
|Dv|_{L^p}^pdy\right)^\frac1p\le L|Dv|_{L^p};& p>1,\\
L|Dv|_{L^1}.&\end{array}\right.$$ The proof is
complete.\end{proof} The second inequality gives us a bound on the
Laplacian of a function integrated against an odd power of the
same function:
\begin{lem}[Laplacian and Odd Power Integral Inequality]
Suppose $v\in C^2((0,L);\R)$, for some $L>0$, with $v(0)=v(L)=0$; then
$$\int_0^L \frac{\partial^2v}{\partial x^2}\cdot v^{2p-1}dx\le-\frac{(2p-1)}{p^2L}|v|^{2p}_{L^{2p}}\mbox{ for all }p\ge1.$$
If   $v\in C^1((0,L);\R)$, for some $L>0$, with $v(0)=0$; then
$$-\int_0^L \frac{\partial v}{\partial x}\cdot \frac{\partial v^{2p-1}}{\partial x}dx\le-\frac{(2p-1)}{p^2L}|v|^{2p}_{L^{2p}}\mbox{ for all }p\ge1.$$
\label{inprd}
\end{lem}
\begin{proof}
Let $u$ be the function given by
$$u(x)\doteq\int_0^x\left(\frac{\partial v^p}{\partial y}\right)^2dy;\ 0\le x\le L.$$
Then $u'(x)=\left(\frac{\partial v^p}{\partial x}\right)^2$ and we have, using \lemref{gPoincare}, that
\begin{eqnarray}
\int_0^L\left(\frac{\partial v^p}{\partial y}\right)^2dy&=|u'|_{L^1}\ge\frac1L|u|_{L^1}
=\frac1L\int_0^L\int_0^x\left(\frac{\partial v^p}{\partial y}\right)^2dydx\\
&\ge\frac1L\int_0^L\left(\int_0^x\frac{\partial v^p}{\partial y}dy\right)^2dx\\
&=\frac1L|v|_{L^{2p}}^{2p}. \label{haha}
\end{eqnarray}
Therefore,
\begin{eqnarray*}
\int_0^L \frac{\partial^2v}{\partial x^2}\cdot v^{2p-1}dx&=-\int_0^L \frac{\partial v}{\partial x}\cdot \frac{\partial v^{2p-1}}{\partial x}dx
=-(2p-1)\int_0^L\left(v^{p-1}\cdot\frac{\partial v}{\partial x}\right)^2dx\\
&=-\frac{2p-1}{p^2}\int_0^L\left(\frac{\partial v^p}{\partial
y}\right)^2dy\le-\frac{2p-1}{p^2L}|v|_{L^{2p}}^{2p}.
\end{eqnarray*}
where the last inequality follows from \eqnref{haha}.
\end{proof}
We now give a probabilistic proof of \eqnref{70} in the case of
the heat equation on $\Rd$; i.e., when $[0,L]$ is replaced with
$\Rd$ and $G_t(x,y)$ is replaced with the fundamental solution to
the heat equation on $\Rd$, $p_t(x,y)$.
\begin{lem}
With the notations above, we have
\begin{equation}
\int_\Rd u_0(y)\frac{\partial}{\partial x}p_t(x,y)dy\le K.
\label{707}
\end{equation}
for some universal constant $K>0$ whenever $u_0$ is Lipschitz.
\label{derbd}
\end{lem}
\begin{proof}
Let $B^x=\left\{B^x_t\doteq\sqrt2\tilde{B}^{x/\sqrt2}_t;0\le
t<\infty\right\}$, where $\tilde{B}^x=\left\{\tilde{B}^x_t; 0\le
t<\infty\right\}$ is a standard $d$-dimensional Brownian motion
starting at $x\in\Rd$. Then, $p_t(x,y)$ is the density of the
scaled Brownian motion $B^x$ on $\Rd$, $p_t(x,y)$, we have
\begin{eqnarray*}
\left|D_j M(t,x)\right|=\left|\int_\Rd
u_0(y)\frac{\partial}{\partial x_j}p_t(x,y)dy\right|
=\left|\int_\Rd u_0(y)\frac{-\left(x_j-y_j\right)}{2t}(4\pi
t)^{-d/2}e^{-|x-y|^2/4t}dy \right|
\\=\left|-\frac{1}{2t}\E\left[(x_j-B_t^{j,x})u_0(B_t^x)\right]\right|\le
\frac1t\E\left|(x_j-B_t^{j,x})\left(u_0(B_t^x)-u_0(x)\right)\right|
\\\le\frac1t\left[\E(x_j-B_t^{j,x})^2\E(u_0(B_t^x)-u_0(x))^2\right]^{1/2}
\\\le\frac{K_1}{t}\left[\E(x_j-B_t^{j,x})^2\E\left|B_t^x-x\right|^2\right]^{1/2}\le K_2,
\label{140}
\end{eqnarray*}
where $D_j=\partial/\partial x_j$ and $B^{j,x}$ is the $j$-th
component of the $d$-dimensional $B^x$, $1\le j\le d$; and where
we have used elementary facts about the Brownian motion $B^x$,
H\"older inequality, and the Lipschitz condition on $u_0$ to get
\eqnref{140}.
\end{proof}

\end{document}